\documentclass[12pt, reqno]{amsart}
\usepackage{ amsmath,amsthm, amscd, amsfonts, amssymb, graphicx, color}
\usepackage[bookmarksnumbered, colorlinks, plainpages]{hyperref}
\newtheorem{thm}{Theorem}[section]
\newtheorem{cor}[thm]{Corollary}
\newtheorem{lem}[thm]{Lemma}

\numberwithin{equation}{section}

\begin{document}

\title{the g-Drazin invertibility in a Banach algebra}

\author{Huanyin Chen}
\author{Marjan Sheibani$^*$}
\address{
Department of Mathematics\\ Hangzhou Normal University\\ Hang -zhou, China}
\email{<huanyinchenhz@163.com>}
\address{Farzanegan Campus, Semnan University, Semnan, Iran}
\email{<m.sheibani@semnan.ac.ir>}


\begin{abstract} We present necessary and sufficient conditions under which the anti-triangular matrix $\left(
  \begin{array}{cc}
    a&b\\
    1&0
  \end{array}
\right)$ over a Banach algebra has g-Drazin inverse. New additive results for g-Drazin inverse are obtained. Then we apply our results to $2\times 2$ operator matrices and generalize many known results, e.g.,~\cite[Theorem 2.2]{D}, ~\cite[Theorem 2.1]{YL} and ~\cite[Theorem 4.1]{Y}.
\end{abstract}

\thanks{Corresponding author: Marjan Sheibani}

\maketitle

\section{Introduction}

Let $\mathcal{A}$ be a Banach algebra with an identity. An element $a$ in $\mathcal{A}$ has g-Drazin inverse provided that there exists some $b\in \mathcal{A}$ such that $$b=bab, ab=ba, a-a^2b\in \mathcal{A}^{qnil}.$$ Here, $\mathcal{A}^{qnil}=\{a\in \mathcal{A}~|~1+ax\in \mathcal{A}^{-1}~\mbox{whenever}~ax=xa\}$. That is,
$x\in \mathcal{A}^{qnil}\Leftrightarrow \lim\limits_{n\to\infty}\parallel x^n\parallel^{\frac{1}{n}}=0.$ We say that $a\in \mathcal{A}$ has Drazin inverse $a^D$ if $\mathcal{A}^{qnil}$ is replaced by the set $\mathcal{A}^{nil}$ of all nilpotents in $\mathcal{A}$. As is well known, $a\in \mathcal{A}$ has Drazin (g-Drazin) inverse if and only if there exists an idempotent $e\in \mathcal{A}$ such that $ae=ea, a-e\in \mathcal{A}$ is invertible and $ae\in \mathcal{A}^{nil} (\mathcal{A}^{qnil}).$ The Drazin and g-Drazin inverses play important roles in matrix and operator theory. They also were extensively studied in ring theory under strongly $\pi$-regularity and quasipolarity (see~\cite{CC,CC1,D,S,XS,Z,ZM,ZC}).

The solutions to singular systems of differential equations are determined by the the Drazin (g-Drazin) inverses of certain anti-triangular block complex matrices (see~\cite{B}). This inspires to investigate the Drazin (g-Drazin) invertibility for the anti-triangular matrix $M=\left(
\begin{array}{cc}
a&1\\
b&0
\end{array}
\right)\in M_2({\mathcal{A}})$ with $a,b\in \mathcal{A}^d$. In ~\cite{P}, Patr\'icio and Hartwig considered the case $ab=ba$ for the Drazin inverse of $M$. For bounded linear operators on Banach spaces, a new expression of $M^D$ was given under the same condition (see~\cite[Theorem 3.8]{XS}). Also Bu et al. gave the alternative representation of $M^D$ for subblock complex matrices.
In ~\cite[Theorem 2.3]{ZM}, Zhang and Mosi\'c presented the g-Drazin inverse of $M$ under the condition $bab^{\pi}=0.$
In~\cite[Theorem 2.6]{Z} the g-Drazin inverse $M^d$ under the conditions $b^dab^{\pi}=0, b^{\pi}ba=0$ has been investigated. For the anti-triangular operator matrix $M$ over a complex Hilbert space, Yu and Deng characterized its Drazin inverse under wider conditions $b^{\pi}ab^D=0, b^{\pi}ab=b^{\pi}ba$ and $b^{\pi}ab^D=0, b^{\pi}ab(b^{\pi}a)^{\pi}=0, (b^{\pi}a)^Db^{\pi}ab=0$ (see~\cite[Theorem 4.1]{Y}). These conditions were also considered in~\cite[Theorem 2.12]{ZCM}.

In Section 2, we present necessary and sufficient conditions under which the anti-triangular matrix $\left(
  \begin{array}{cc}
    a&1\\
    b&0
  \end{array}
\right)$ over a Banach algebra has g-Drazin inverse. ~\cite[Theorem 4.1]{Y} and ~\cite[Theorem 2.12]{ZCM} are thereby extended to a more general setting.

Let $a,b\in \mathcal{A}^d$. Many authors have studied when $a+b\in \mathcal{A}$ has Drazin (g-Drazin) inverse. In~\cite[Theorem 2.1]{Y}, Yang and Liu
considered the conditions $ab^2=0$ and $aba=0$. In ~\cite[Theorem 2.4]{CM1}, the authors extend to the conditions $ab^2=0$ and $b^{\pi}aba=0$.
In ~\cite[Theorem 3.1]{S}, for the setting of complex matrices, Shakoor et al. investigated the Drazin inverse of $a+b$ under the conditions
$ab^2=0$ and $a^2ba=0$. These conditions were also considered in ~\cite[Theorem 3.1]{SZ}. We refer the reader for more related papers, e.g.,~\cite{X,YL,ZM,ZC}.

In Section 3, we apply our results to establish some new additive results. Let $a,b,ab\in \mathcal{A}^d$. If $ab^2=0,(ab)^{\pi}a(ab)^d=0~\mbox{and}~(ab)^{\pi}aba=0,$ we prove that $a+b\in \mathcal{A}^d$. This also extends the existing results above.

Let $X,Y$ be Banach spaces and $M=\left(
\begin{array}{cc}
A&B\\
C&D
\end{array}
\right)~~~~~(*)$, where $A\in \mathcal{L}(X), B\in \mathcal{L}(X,Y), C\in \mathcal{L}(Y, X)$ and $D\in \mathcal{L}(Y)$. Then $M$ is a bounded linear operator on $X\oplus Y$. Finally, in the last section, we split $M$ into the sum of two block operator matrices. We then establish new results for the g-Drazin inverse of $2\times 2$ block operator matrix $M$. These also recover some known results, e.g., ~\cite[Theorem 2.2]{D}.

Throughout the paper, we use $\mathcal{A}^{d}$ to denote the set of all g-Drazin invertible elements in $\mathcal{A}$. Let $a\in \mathcal{A}^d$. The spectral idempotent $1-aa^d$ is denoted by $a^{\pi}$. $\mathcal{L}(X)$ denotes the Banach algebra of all bounded linear operators on the Banach space $X$. ${\Bbb C}^{n\times n}$ stands for the Banach algebra of all $n\times n$ complex matrices.

\section{anti-triangular matrices over Banach algebra}

The aim of this section is to investigate the g-Drazin invert of the operator matrix $\left(
  \begin{array}{cc}
    a&b\\
    1&0
  \end{array}
\right)$ over a Banach algebra $\mathcal{A}.$ We begin with

\begin{lem} (see~\cite[Lemma 1.3]{ZM}) Let $a,b\in \mathcal{A}^d$. If $ab=0$, then $a+b\in \mathcal{A}^d$ and $$(a+b)^d=\sum\limits_{i=0}^{\infty}b^ib^{\pi}(a^d)^{i+1}+\sum\limits_{i=0}^{\infty}(b^d)^{i+1}a^ia^{\pi}.$$
\end{lem}

\begin{thm} Let $a,b,b^{\pi}a\in \mathcal{A}^d$ and $b^{\pi}ab^d=0$. Then the following are equivalent:\end{thm}
\begin{enumerate}
\item [(1)] $\left(
  \begin{array}{cc}
    a&1\\
    b&0
  \end{array}
\right)$ has g-Drazin inverse.
\item [(2)] $\left(
  \begin{array}{cc}
    b^{\pi}a&1\\
    b^{\pi}b&0
  \end{array}
\right)$ has g-Drazin inverse.
\item [(3)] $\left(
  \begin{array}{cc}
    ab^{\pi}&1\\
    bb^{\pi}&0
  \end{array}
\right)$ has g-Drazin inverse.
\end{enumerate}
\begin{proof} $(1)\Rightarrow (2)$ Let $M=\left(
  \begin{array}{cc}
    a&1\\
   b&0
  \end{array}
\right)$ and $p=\left(
\begin{array}{cc}
b^{\pi}&0\\
0&b^{\pi}
\end{array}
\right).$ Since $b^{\pi}ab^d=0$, we have $$pM(1-p)=0, M(1-p)=\left(
\begin{array}{cc}
bb^dabb^d&bb^d\\
b^2b^d&0
\end{array}
\right).$$ We directly verify that $[M(1-p)]^{\#}=\left(
\begin{array}{cc}
0&-bb^d\\
-b^2b^d&bb^dabb^d
\end{array}
\right)$. Since $M$ has g-Drazin inverse, it follows by~\cite[Lemma 2.2]{ZM} that $pM=\left(
\begin{array}{cc}
b^{\pi}a&b^{\pi}\\
b^{\pi}b&0
\end{array}
\right)$ has g-Drazin inverse.

Let $N=\left(
  \begin{array}{cc}
    b^{\pi}a&1\\
    b^{\pi}b&0
  \end{array}
\right)$. It is easy to verify that
$$\begin{array}{rll}
N&=&\left(
\begin{array}{cc}
1&0\\
0&b^{\pi}
\end{array}
\right)\left(
\begin{array}{cc}
b^{\pi}a&1\\
b^{\pi}b&0
\end{array}
\right),\\
\left(
\begin{array}{cc}
b^{\pi}a&b^{\pi}\\
b^{\pi}b&0
\end{array}
\right)&=&\left(
\begin{array}{cc}
b^{\pi}a&1\\
b^{\pi}b&0
\end{array}
\right)\left(
\begin{array}{cc}
1&0\\
0&b^{\pi}
\end{array}
\right).\\
\end{array}$$
By using Cline's formula (see~\cite[Theorem 2.1]{LC}), $N$ has g-Drazin inverse.

$(2)\Rightarrow (1) $ Let $e=\left(
\begin{array}{cc}
bb^d&0\\
0&1
\end{array}
\right)$. Then $M=\left(
\begin{array}{cc}
\alpha&\beta\\
\gamma&\delta
\end{array}
\right)_e.$ Here, $$\begin{array}{c}
\alpha=\left(
\begin{array}{cc}
bb^dabb^d&bb^d\\
b^2b^d&0
\end{array}
\right),\beta=\left(
\begin{array}{cc}
bb^dab^{\pi}&0\\
bb^{\pi}&0
\end{array}
\right),\\
\gamma=\left(
\begin{array}{cc}
0&b^{\pi}\\
0&0
\end{array}
\right),\delta=\left(
\begin{array}{cc}
b^{\pi}a&0\\
0&0
\end{array}
\right).
\end{array}$$ Then
$$\begin{array}{c}
\alpha^{\#}=\left(
\begin{array}{cc}
0&-bb^d\\
-b^2b^d&bb^dabb^d
\end{array}
\right), \alpha^{\pi}=\left(
\begin{array}{cc}
b^{\pi}&0\\
0&b^{\pi}
\end{array}
\right);\\
\beta+\gamma+\delta=\left(
\begin{array}{cc}
bb^dab^{\pi}+b^{\pi}a&b^{\pi}\\
bb^{\pi}&0
\end{array}
\right).
\end{array}$$ Hence we have $$(\beta+\gamma+\delta)\alpha=\left(
\begin{array}{cc}
bb^dab^{\pi}+b^{\pi}a&b^{\pi}\\
bb^{\pi}&0
\end{array}
\right)\left(
\begin{array}{cc}
bb^dabb^d&bb^d\\
b^2b^d&0
\end{array}
\right)=0.$$ One easily checks that
$$\begin{array}{rll}
\beta+\gamma+\delta&=&\left(
\begin{array}{cc}
bb^da+b^{\pi}a&1\\
b&0
\end{array}
\right)\left(
\begin{array}{cc}
b^{\pi}&0\\
0&b^{\pi}
\end{array}
\right);\\
\left(
\begin{array}{cc}
b^{\pi}a&b^{\pi}\\
b^{\pi}b&0
\end{array}
\right)&=&\left(
\begin{array}{cc}
b^{\pi}&0\\
0&b^{\pi}
\end{array}
\right)\left(
\begin{array}{cc}
bb^da+b^{\pi}a&1\\
b&0
\end{array}
\right);\\
\left(
\begin{array}{cc}
b^{\pi}a&b^{\pi}\\
b^{\pi}b&0
\end{array}
\right)&=&\left(
\begin{array}{cc}
b^{\pi}a&1\\
b^{\pi}b&0
\end{array}
\right)\left(
\begin{array}{cc}
1&0\\
0&b^{\pi}
\end{array}
\right),\\
\left(
\begin{array}{cc}
b^{\pi}a&1\\
b^{\pi}b&0
\end{array}
\right)&=&\left(
\begin{array}{cc}
1&0\\
0&b^{\pi}
\end{array}
\right)\left(
\begin{array}{cc}
b^{\pi}a&1\\
b^{\pi}b&0
\end{array}
\right).\\
\end{array}$$
Since $N$ has g-Drazin inverse, by using Cline's formula, $\beta+\gamma+\delta$ has g-Drazin inverse.
Set $N=\left(
  \begin{array}{cc}
    b^{\pi}a&1\\
    b^{\pi}b&0
  \end{array}
\right)$. Then we have
$$\begin{array}{lll}
\left(
\begin{array}{cc}
b^{\pi}a&b^{\pi}\\
b^{\pi}b&0
\end{array}
\right)^d&=&\left(
\begin{array}{cc}
b^{\pi}a&1\\
b^{\pi}b&0
\end{array}
\right)(N^d)^2\left(
\begin{array}{cc}
1&0\\
0&b^{\pi}
\end{array}
\right)\\
&=&N^d\left(
\begin{array}{cc}
1&0\\
0&b^{\pi}
\end{array}
\right).
\end{array}$$
Therefore $$\begin{array}{ll}
&(\beta+\gamma+\delta)^d\\
=&\left(
\begin{array}{cc}
bb^da+b^{\pi}a&1\\
b&0
\end{array}
\right)\big[\left(
\begin{array}{cc}
b^{\pi}a&b^{\pi}\\
b^{\pi}b&0
\end{array}
\right)^d\big]^2\left(
\begin{array}{cc}
b^{\pi}&0\\
0&b^{\pi}
\end{array}
\right)\\
=&\left(
\begin{array}{cc}
bb^da+b^{\pi}a&1\\
b&0
\end{array}
\right)N^d\left(
\begin{array}{cc}
1&0\\
0&b^{\pi}
\end{array}
\right)N^d\left(
\begin{array}{cc}
b^{\pi}&0\\
0&b^{\pi}
\end{array}
\right).\\
\end{array}$$
In light of Lemma 2.1, $M$ has g-Drazin inverse. In fact, we get
$$\begin{array}{rll}
M^d&=&\sum\limits_{i=0}^{\infty}\alpha^i\alpha^{\pi}[(\beta+\gamma+\delta)^d]^{i+1}\\
&+&\sum\limits_{i=0}^{\infty}[\alpha^{\#}]^{i+1}(\beta+\gamma+\delta)^i(\beta+\gamma+\delta)^{\pi}\\
&=&\alpha^{\pi}(\beta+\gamma+\delta)^d\\
&+&\sum\limits_{i=0}^{\infty}[\alpha^{\#}]^{i+1}(\beta+\gamma+\delta)^i(\beta+\gamma+\delta)^{\pi}.\\
\end{array}$$

$(2)\Leftrightarrow (3)$ One directly checks that
$$\begin{array}{rll}
\left(
  \begin{array}{cc}
    b^{\pi}a&1\\
    b^{\pi}b&0
  \end{array}
\right)&=&\left(
  \begin{array}{cc}
    1&0\\
    0&b^{\pi}
  \end{array}
\right)\left(
  \begin{array}{cc}
    b^{\pi}a&1\\
    b^{\pi}b&0
  \end{array}
\right),\\
\left(
  \begin{array}{cc}
    b^{\pi}a&b^{\pi}\\
    b^{\pi}b&0
  \end{array}
\right)&=&\left(
  \begin{array}{cc}
    b^{\pi}a&1\\
    b^{\pi}b&0
  \end{array}
\right)\left(
  \begin{array}{cc}
    1&0\\
    0&b^{\pi}
  \end{array}
\right);\\
\left(
  \begin{array}{cc}
    b^{\pi}a&b^{\pi}\\
    b^{\pi}b&0
  \end{array}
\right)&=&\left(
  \begin{array}{cc}
    b^{\pi}&0\\
    0&b^{\pi}
  \end{array}
\right)\left(
  \begin{array}{cc}
    a&1\\
    b&0
  \end{array}
\right),\\
\left(
  \begin{array}{cc}
    ab^{\pi}&b^{\pi}\\
    bb^{\pi}&0
  \end{array}
\right)&=&\left(
  \begin{array}{cc}
    a&1\\
    b&0
  \end{array}
\right)\left(
  \begin{array}{cc}
    b^{\pi}&0\\
    0&b^{\pi}
  \end{array}
\right);\\
\left(
  \begin{array}{cc}
     ab^{\pi}&b^{\pi}\\
    bb^{\pi}&0
  \end{array}
\right)&=&\left(
  \begin{array}{cc}
    ab^{\pi}&1\\
    bb^{\pi}&0
  \end{array}
\right)\left(
  \begin{array}{cc}
   1&0\\
   0&b^{\pi}
  \end{array}
\right),\\
\left(
  \begin{array}{cc}
    ab^{\pi}&1\\
    bb^{\pi}&0
  \end{array}
\right)&=&\left(
  \begin{array}{cc}
   1&0\\
   0&b^{\pi}
  \end{array}
\right)\left(
  \begin{array}{cc}
    ab^{\pi}&1\\
    bb^{\pi}&0
  \end{array}
\right).
\end{array}$$
Therefore we complete the proof by repeatedly using Cline's formula (see~\cite[Theorem 2.1]{LC}).\end{proof}

\begin{cor} Let $a,b, b^{\pi}a\in \mathcal{A}^d$. If $b^{\pi}ab^d=0, abb^{\pi}=b^{\pi}ba,$ then $\left(
  \begin{array}{cc}
    a&1\\
    b&0
  \end{array}
\right)$ has g-Drazin inverse.\end{cor}
\begin{proof} Let $N=\left(
  \begin{array}{cc}
    ab^{\pi}&1\\
    bb^{\pi}&0
  \end{array}
\right)$. We check that $$(ab^{\pi})(bb^{\pi})=abb^{\pi}=b^{\pi}ba=b^{\pi}bab^{\pi}=(bb^{\pi})(ab^{\pi}).$$
Hence $N$ has g-Drazin inverse. This completes the proof by Theorem 2.2.\end{proof}

Yu et al. characterized the Drazin invertibility of an anti-triangular matrix over a complex Hibert space by using solutions of certain operator equations (see~\cite[Theorem 4.1]{Y}). We now generalize their main results to the g-Darzin inverse in a Banach algebra by using ring technique as follows.

\begin{cor} Let $a,b, b^{\pi}a\in \mathcal{A}^d$. If $b^{\pi}ab^d=0, b^{\pi}ab=b^{\pi}ba,$ then $\left(
  \begin{array}{cc}
    a&1\\
    b&0
  \end{array}
\right)$ has g-Drazin inverse.\end{cor}
\begin{proof} Let $N=\left(
  \begin{array}{cc}
    b^{\pi}a&1\\
    b^{\pi}b&0
  \end{array}
\right)$. By hypothesis, we have $$(b^{\pi}a)(b^{\pi}b)=b^{\pi}ab=b^{\pi}ba=(b^{\pi}b)(b^{\pi}a).$$
As in the proof of Theorem 2.2, $N$ has g-Drazin inverse. This completes the proof by
Theorem 2.2.\end{proof}

\begin{cor} Let $a,b, b^{\pi}a\in \mathcal{A}^d$. If $b^{\pi}ab^d=0, (b^{\pi}a)^db^{\pi}ab=0, b^{\pi}ab(b^{\pi}a)^{\pi}=0,$ then $\left(
  \begin{array}{cc}
    a&1\\
    b&0
  \end{array}
\right)$ has g-Drazin inverse.\end{cor}
\begin{proof} Let $N=\left(
  \begin{array}{cc}
    b^{\pi}a&1\\
    b^{\pi}b&0
  \end{array}
\right)$. Then $N^2=\left(
  \begin{array}{cc}
    (b^{\pi}a)^2+b^{\pi}b&b^{\pi}a\\
    (b^{\pi}b)(b^{\pi}a)&b^{\pi}b
  \end{array}
\right)$. Write $N^2=P+Q$, where $$P=\left(
  \begin{array}{cc}
    (b^{\pi}a)^2&b^{\pi}a\\
    0&0
  \end{array}
\right),Q=\left(
  \begin{array}{cc}
    b^{\pi}b&0\\
    (b^{\pi}b)(b^{\pi}a)&b^{\pi}b
  \end{array}
\right).$$ By hypothesis, we have $$P^d=\left(
  \begin{array}{cc}
    [(b^{\pi}a)^d]^2&[(b^{\pi}a)^d]^3\\
    0&0
  \end{array}
\right),P^{\pi}=\left(
  \begin{array}{cc}
    (b^{\pi}a)^{\pi}&-(b^{\pi}a)^d\\
    0&I
  \end{array}
\right).$$ Clearly, $(b^{\pi}a)^db^{\pi}b=(b^{\pi}a)^d[(b^{\pi}a)^db^{\pi}ab]$ and $-(b^{\pi}b)((b^{\pi}a))(b^{\pi}a)^d+b^{\pi}b=bb^{\pi}(b^{\pi}a)^{\pi}$.
 Hence $P^dQ=0, PQP^{\pi}=0.$ Since $P^{\pi}Q=Q$ has g-Drazin inverse, so has $QP^{\pi}$ by Cline's formula.
In view of Lemma 2.1, $PP^{\pi}+QP^{\pi}$ has g-Drazin inverse. Choose $p=PP^d$. Then $$P=\left(
  \begin{array}{cc}
    P^2P^d&0\\
    0&PP^{\pi}
  \end{array}
\right)_p, Q=\left(
  \begin{array}{cc}
    0&0\\
    0&QP^{\pi}
  \end{array}
\right)_p.$$ Hence, $N^2=\left(
  \begin{array}{cc}
    P^2P^d&0\\
    0&PP^{\pi}+QP^{\pi}
  \end{array}
\right)_p$ has g-Drazin inverse, and then $N$ has g-Drazin inverse.
According to Theorem 2.2, we complete the proof.\end{proof}

We are now ready to prove the following:

\begin{thm} Let $a,b,b^{\pi}a\in \mathcal{A}^d$. If $b^{\pi}(ab^2)=0$ and $b^{\pi}(aba)=0$, then $\left(
  \begin{array}{cc}
    a&1\\
    b&0
  \end{array}
\right)$ has g-Drazin inverse.\end{thm}
\begin{proof} Since $b^{\pi}(ab^2)=0$, we have $$b^{\pi}ab^d=b^{\pi}ab^2(b^d)^3=0,b^{\pi}a(b^{\pi}b)^2=0, b^{\pi}a(b^{\pi}b)b^{\pi}a=0.$$
Let $N=\left(
  \begin{array}{cc}
    b^{\pi}a&1\\
    b^{\pi}b&0
  \end{array}
\right).$ Then $N^2=\left(
  \begin{array}{cc}
    (b^{\pi}a)^2+b^{\pi}b&b^{\pi}a\\
    b^{\pi}bb^{\pi}a&b^{\pi}b
  \end{array}
\right).$ Write $N^2=P+Q$, where
$$P=\left(
  \begin{array}{cc}
    (b^{\pi}a)^2&b^{\pi}a\\
    0&0
  \end{array}
\right), Q=\left(
  \begin{array}{cc}
    b^{\pi}b&0\\
    b^{\pi}bb^{\pi}a&b^{\pi}b
  \end{array}
\right).$$ It is easy to verify that $$\begin{array}{c}
PQ^2=\left(
  \begin{array}{cc}
    (b^{\pi}a)^2b^{\pi}b&b^{\pi}ab\\
    0&0
  \end{array}
\right)\left(
  \begin{array}{cc}
    b^{\pi}b&0\\
    b^{\pi}bb^{\pi}a&b^{\pi}b
  \end{array}
\right)=0, \\
PQP=\left(
  \begin{array}{cc}
    (b^{\pi}a)^2b^{\pi}b&b^{\pi}ab\\
    0&0
  \end{array}
\right)\left(
  \begin{array}{cc}
    (b^{\pi}a)^2&b^{\pi}a\\
    0&0
  \end{array}
\right)=0.
\end{array}$$ By virtue of~\cite[Theorem 2.4]{CM1},
$N^2$ has g-Drazin inverse. It follows from~\cite[Corollary 2.2]{M} that $N$ has g-Drazin inverse.
In light of Theorem 2.2, the result follows.\end{proof}

\begin{cor} Let $a,b\in \mathcal{A}^d$. If $ab^2=0$ and $aba=0$, then $\left(
  \begin{array}{cc}
    a&1\\
    b&0
  \end{array}
\right)$ has g-Drazin inverse.\end{cor}
\begin{proof} Since $ab^2=0$, $ab^{\pi}=a-(ab^2)(b^d)^2=a\in \mathcal{A}^d$. By Cline's formula, $b^{\pi}a\in \mathcal{A}^d$.
This completes the proof by Theorem 2.6.\end{proof}

\begin{cor} Let $a,b, b^{\pi}a\in \mathcal{A}^d$. If $b^{\pi}ab=0$, then $\left(
  \begin{array}{cc}
    a&1\\
    b&0
  \end{array}
\right)$ has g-Drazin inverse.\end{cor}
\begin{proof} Since $b^{\pi}ab=0$, $(b^{\pi}a)^2=0$, and so $b^{\pi}a\in \mathcal{A}^d$. So the corollary is true by Theorem 2.6.\end{proof}

\section{additive properties}

In this section we establish some elementary additive properties of g-Drazin inverse in a Banach algebra. The following fact will also be used in our subsequent investigations.

\begin{thm} Let $a,b,ab\in \mathcal{A}^d$. If $ab^2=0,(ab)^{\pi}a(ab)^d=0~\mbox{and}~(ab)^{\pi}$ $aba=0,$ then $a+b\in \mathcal{A}^d$.\end{thm}
\begin{proof} Obviously, we have $a+b=(1,b)\left(
\begin{array}{c}
a\\
1
\end{array}
\right)$. In view of Cline's formula, it suffices to prove
$$M=\left(
\begin{array}{c}
a\\
1
\end{array}
\right)(1,b)=\left(
\begin{array}{cc}
a&ab\\
1&b
\end{array}
\right)$$  has g-Drazin inverse. Write $M=K+L,$ where $$K=\left(
\begin{array}{cc}
a&ab\\
1&0
\end{array}
\right), L=\left(
\begin{array}{cc}
0&0\\
0&b
\end{array}
\right).$$
Let $H=\left(
\begin{array}{cc}
a&1\\
ab&0
\end{array}
\right)$ and $N=\left(
\begin{array}{cc}
(ab)^{\pi}a&1\\
(ab)^{\pi}ab&0
\end{array}
\right).$ One easily checks that
$$\begin{array}{c}
(ab)^{\pi}a[(ab)^{\pi}ab]^2=(ab)^{\pi}a(ab)^{\pi}(aba)b=0,\\
(ab)^{\pi}a[(ab)^{\pi}ab](ab)^{\pi}a=(ab)^{\pi}a(ab)^{\pi}(aba)=0.
\end{array}$$ In light of Corollary 2.7, $N$ has g-Drazin inverse. By hypothesis, $(ab)^{\pi}a(ab)^d=0$. According to Theorem 2.2, $H$ has g-Drazin inverse.
Clearly, $$H=\left(
\begin{array}{cc}
1&0\\
0&ab
\end{array}
\right)\left(
\begin{array}{cc}
a&1\\
1&0
\end{array}
\right), K=\left(
\begin{array}{cc}
a&1\\
1&0
\end{array}
\right)\left(
\begin{array}{cc}
1&0\\
0&ab
\end{array}
\right).$$ By using Cline's formula, $K$ has g-Drazin inverse. Since $ab^2=0$, we have $KL=0$. In light of Lemma 2.1, $M$ has g-Drazin inverse. Therefore $a+b\in \mathcal{A}^d$.\end{proof}

\begin{cor} Let $a,b,ab\in \mathcal{A}^d$. If $a^2b=0, (ab)^db(ab)^{\pi}=0~\mbox{and}~bab$ $(ab)^{\pi}=0,$ then $a+b\in \mathcal{A}^d$.
\end{cor}
\begin{proof} Since $(\mathcal{A},\cdot)$ is a Banach algebra, $(\mathcal{A},*)$ is a Banach algebra with the multiplication $a*b=b\cdot a$. Then we complete the proof by applying Theorem 3.1 to the Banach algebra $(\mathcal{A},*)$.\end{proof}

We are now ready to generalize ~\cite[Theorem 3.1]{S} as follow:

\begin{thm} Let $a,b,ab\in \mathcal{A}^d$. If $ab^2=0$ and $(ab)^{\pi}a^2ba=0,$ then $a+b\in \mathcal{A}^d$.
\end{thm}
\begin{proof} Let $M=\left(
\begin{array}{cc}
a&ab\\
1&b
\end{array}
\right).$ Write $M=K+L,$ where $$K=\left(
\begin{array}{cc}
a&ab\\
1&0
\end{array}
\right), L=\left(
\begin{array}{cc}
0&0\\
0&b
\end{array}
\right).$$
Let $H=\left(
\begin{array}{cc}
a&1\\
ab&0
\end{array}
\right).$ By hypothesis, we check that
$$(ab)^{\pi}a(ab)^2=0,(ab)^{\pi}a(ab)a=0.$$
According to Theorem 2.6, $H$ is g-Drazin inverse. As in the proof of Theorem 3.1, by using Cline's formula,
$K$ has g-Drazin inverse. Since $ab^2=0$, it follows by Lemma 2.1 that $M$ has g-Drazin inverse.
Observing that $$\begin{array}{rll}
a+b&=&(1,b)\left(
\begin{array}{c}
a\\
1
\end{array}
\right),\\
M&=&\left(
\begin{array}{c}
a\\
1
\end{array}
\right)(1,b),
\end{array}$$ by using Cline's formula again, $a+b$ has g-Drazin inverse.\end{proof}

\begin{cor} Let $a,b,ab\in \mathcal{A}^d$. If $a^2b=0$ and $bab^2(ab)^{\pi}=0,$ then $a+b\in \mathcal{A}^d$.
\end{cor}
\begin{proof} Similarly to Corollary 3.2, we obtain the result by Theorem 3.3.\end{proof}

\begin{cor} Let $a,b,ab\in \mathcal{A}^d$. If $ab^2=0$ and $a^2ba=0,$  then $a+b\in \mathcal{A}^d$.
\end{cor}
\begin{proof} This is obvious by Theorem 3.3.\end{proof}

\section{operator matrices over Banach spaces}

In this section we apply our results to establish g-Drazin invertibility for the block operator matrix $M$ as in $(*)$. Throughout this section, we always assume that $A,D,BC\in \mathcal{L}(X)^d$. We come now to extend ~\cite[Theorem 3.1]{Y} as follows.

\begin{thm} If $(BC)^{\pi}ABCA=0, (BC)^{\pi}ABCB=0, DCA=0$ and $DCB=0$, then $M$ has g-Drazin inverse.
\end{thm}
\begin{proof}  Write $M=P+Q$, where
$$P=\left(
\begin{array}{cc}
A&B\\
0&D
\end{array}
\right), Q=\left(
\begin{array}{cc}
0&0\\
C&0
\end{array}
\right).$$ Clearly, $Q^2=0$, and so $PQ^2=0$. Moreover, we have
$$\begin{array}{rll}
PQ&=&\left(
\begin{array}{cc}
BC&0\\
DC&0
\end{array}
\right),\\
(PQ)^d&=&\left(
\begin{array}{cc}
(BC)^d&0\\
DC[(BC)^d]^2&0
\end{array}
\right),\\
(PQ)^{\pi}&=&\left(
\begin{array}{cc}
(BC)^{\pi}&0\\
-DC(BC)^d&I
\end{array}
\right).
\end{array}$$
We easily check that
$$\begin{array}{ll}
&(PQ)^{\pi}P^2QP\\
=&\left(
\begin{array}{cc}
(BC)^{\pi}&0\\
-DC(BC)^d&I
\end{array}
\right)\left(
\begin{array}{cc}
A&B\\
0&D
\end{array}
\right)\left(
\begin{array}{cc}
BC&0\\
DC&0
\end{array}
\right)\left(
\begin{array}{cc}
A&B\\
0&D
\end{array}
\right)\\
=&\left(
\begin{array}{cc}
(BC)^{\pi}A&(BC)^{\pi}B\\
0&D
\end{array}
\right)\left(
\begin{array}{cc}
BCA&BCB\\
0&0
\end{array}
\right)\\
=&0.
\end{array}$$
Therefore we complete the proof by Theorem 3.3.\end{proof}

\begin{cor} If $(BC)^{\pi}ABCA=0, (BC)^{\pi}ABCB=0, BDC=0$ and $BD^2=0$, then $M$ has g-Drazin inverse.\end{cor}
\begin{proof} Write $M=P+Q$, where
$$P=\left(
\begin{array}{cc}
A&B\\
C&0
\end{array}
\right), Q=\left(
  \begin{array}{cc}
    0 & 0 \\
    0 & D
  \end{array}
\right).$$ In light of Theorem 4.1, $P$ has g-Drazin inverse. Since $PQ^2=0$ and $PQP=0$, we complete the proof by~\cite[Theorem 2.4]{CM1}.\end{proof}

\begin{thm} If $(BC)^{\pi}A(BC)^d=0, (BC)^{\pi}BCA=0, (BC)^{\pi}$ $BCB=0, DCA=0$ and $DCB=0$, then $M$ has g-Drazin inverse.\end{thm}
\begin{proof} Write $M=P+Q$, where
$$P=\left(
\begin{array}{cc}
A&B\\
0&D
\end{array}
\right), Q=\left(
\begin{array}{cc}
0&0\\
C&0
\end{array}
\right).$$ Then $Q^2=0$ and $PQ=\left(
\begin{array}{cc}
BC&0\\
DC&0
\end{array}
\right);$ hence,
$$(PQ)^d=\left(
\begin{array}{cc}
(BC)^d&0\\
DC[(BC)^d]^2&0
\end{array}
\right), (PQ)^{\pi}=\left(
\begin{array}{cc}
(BC)^{\pi}&0\\
-DC(BC)^d&I
\end{array}
\right).$$
By hypothesis, we verify that
$$\begin{array}{ll}
&(PQ)^{\pi}P(PQ)^d\\
=&\left(
\begin{array}{cc}
(BC)^{\pi}&0\\
-DC(BC)^d&1
\end{array}
\right)\left(
\begin{array}{cc}
A&B\\
0&D
\end{array}
\right)\left(
\begin{array}{cc}
(BC)^d&0\\
DC[(BC)^d]^2&0
\end{array}
\right)\\
=&\left(
\begin{array}{cc}
(BC)^{\pi}&0\\
-DC(BC)^d&I
\end{array}
\right)\left(
\begin{array}{cc}
A(BC)^d&0\\
D^2C[(BC)^d]^2&0
\end{array}
\right)\\
=&0,\\
&(PQ)^{\pi}PQP\\
=&\left(
\begin{array}{cc}
(BC)^{\pi}&0\\
-DC(BC)^d&I
\end{array}
\right)\left(
\begin{array}{cc}
BC&0\\
DC&0
\end{array}
\right)\left(
\begin{array}{cc}
A&B\\
0&D
\end{array}
\right)\\
=&\left(
\begin{array}{cc}
(BC)^{\pi}BCA&(BC)^{\pi}BCB\\
DC(BC)^{\pi}A&DC(BC)^{\pi}B
\end{array}
\right)\\
=&0.
\end{array}$$
This completes the proof by Theorem 3.1.\end{proof}

\begin{cor} If $(BC)^{\pi}A(BC)^d=0,$ $(BC)^{\pi}BCA=0, (BC)^{\pi}$ $BCB=0, BDC=0$ and $BD^2=0$, then $M$ has g-Drazin inverse.\end{cor}\begin{proof} Write $M=P+Q$,  where
$$P=\left(
\begin{array}{cc}
A&B\\
C&0
\end{array}
\right), Q=\left(
  \begin{array}{cc}
    0 & 0 \\
    0 & D
  \end{array}
\right).$$ In light of Theorem 4.3, $P$ is g-Drazin inverse. As in the proof of Corollary 4.2,
$M$ has g-Drazin inverse.\end{proof}

\begin{thm} If $(CB)^{\pi}CABC=0,A(BC)^{\pi}ABC=0,ABD=0$ and $CBD=0$, then $M$ has g-Drazin inverse.\end{thm}
\begin{proof} Write $M=P+Q,$ where $$P=\left(
\begin{array}{cc}
A&0\\
C&0
\end{array}
\right), Q=\left(
\begin{array}{cc}
0&B\\
0&D
\end{array}
\right).$$ Then $$
\begin{array}{rll}
PQ^2&=&\left(
\begin{array}{cc}
A&0\\
C&0
\end{array}
\right)\left(
\begin{array}{cc}
0&B\\
0&D
\end{array}
\right)^2\\
&=&\left(
\begin{array}{cc}
0&ABD\\
0&CBD
\end{array}
\right),\\
(PQ)^{\pi}P^2QP&=&\left(
\begin{array}{cc}
0&AB\\
0&CB
\end{array}
\right)^{\pi}\left(
\begin{array}{cc}
A&0\\
C&0
\end{array}
\right)^2\left(
\begin{array}{cc}
0&B\\
0&D
\end{array}
\right)\left(
\begin{array}{cc}
A&0\\
C&0
\end{array}
\right)\\
&=&\left(
\begin{array}{cc}
I&-AB(CB)^d\\
0&(CB)^{\pi}
\end{array}
\right)\left(
\begin{array}{cc}
A^2BC&0\\
CABC&0
\end{array}
\right).
\end{array}$$ Clearly, $P$ and $Q$ have g-Drazin inverses. Moreover, $PQ^2=0$ and $Q^{\pi}P^2QP=0$, and therefore we complete the proof by Theorem 3.3.\end{proof}

We now generalize ~\cite[Theorem 2.2]{D}  as follow.

\begin{cor} If $(CB)^{\pi}CABC=0, A(BC)^{\pi}ABC=0, BDC=0$ and $BD^2=0$, then $M$ has g-Drazin inverse.
\end{cor}
\begin{proof} As in the proof of Corollary 4.2, we are through by Theorem 4.5.\end{proof}


\vskip10mm\begin{thebibliography}{99} \bibitem{B} C. Bu; K. Zhang and J. Zhao, Representation of the Drazin inverse on solution of a class singular differential equations, {\it Linear Multilinear Algebra}, {\bf 59}(2011), 863--877.

\bibitem{CM1} H. Chen and M. Sheibani, G-Drazin inverses for operator matrices, {\it Operators and Matrices}, {\bf 14}(2020), 23--31.

\bibitem{CC} J. Cui and J. Chen, When is a
$2\times 2$ matrix ring over a commutative local ring quasipolar?
{\it Comm. Algebra}, {\bf 39}(2011), 3212--3221.

\bibitem{CC1} J. Cui and J. Chen, Quasipolar
triangular matrix rings over local rings, {\it Comm. Algebra},
{\bf 40}(2012), 784--794.

\bibitem{D} E. Dopazo and M.F. Mart\'inez-Serrano, Further results on the representation of the Drazin inverse of a $2\times 2$ block matrix,
{\it Linear Algebra Appl.}, {\bf 432}(2010), 1896--1904.

\bibitem{LC} Y. Liao; J. Chen and J. Cui, Cline's formula for generalized Drazin inverse, {\it Bull. Malays. Math. Sci. Soc.}, {bf 37}(2014), 37-42.

\bibitem{M} D. Mosi\'c, A note on Cline's formula for the generalized Drazin inverse, {\it Linear Multilinear Algebra},
{\bf 63}(2014), 1106--1110.

\bibitem{P} P. Patric\'io and R.E. Hartwig, The $(2,2,0)$ Drazin inverse problem, {\it Linear Algebra Appl.}, {\bf 437}(2012), 2755--2772.

\bibitem{S} A. Shakoor; I. Ali; S. Wali and A. Rehman, Some formulas on the Drazin inverse for the sum of two matrices and block matrices,
{\it Bull. Iran. Math. Soc.}, 2021, https://doi.org/10.1007/s41980-020-00521-3.

\bibitem{SZ} L. Sun; B. Zheng; S. Bai and C. Bu, Formulas for the Drazin inverse of matrices over skew fields,
{\it Filomat}, {\bf 30}(2016), 3377--3388.

\bibitem{X} L. Xia and B. Deng, The Drazin inverse of the sum of two matrices and its applications,
 {\it Filomat}, {\bf 31} (2017), 5151-5158.

\bibitem{XS} Q. Xu; C. Song and L. Zhang, Solvability of certain quadratic operator equations and representations of Drazin inverse,
{\it Linear Algebra Appl.}, {\bf 439} (2013), 291--309.

\bibitem{YL} H. Yang and X. Liu, The Drazin inverse of the sum of two matrices and its applications,
{\it J. Comput. Applied Math.}, {\bf 235}(2011), 1412--1417.

\bibitem{Y} A. Yu; X. Wang and C. Deng, On the Drazin inverse of an anti-triangular block matrix,
{\it . Linear Algebra Appl.}, {\bf 489}(2016), 274--287.

\bibitem{Z} D. Zhang, Representations for generalized Drazin inverse of operator matrices over a Banach space,
{\it Turk. J. Math.}, {\bf 40}(2016), 428--437.

\bibitem{ZM} D. Zhang and D. Mosi\'{c}, Explicit formulae for the generalized Drazin inverse of block matrices over a Banach algebra,
{\it Filomat}, {\bf 32}(2018), 5907--5917.

\bibitem{ZCM} H. Zou; J. Chen and D. Mosi\'c, The Drazin invertibility of an anti-triangular matrix over a ring, {\it Studia Scient. Math. Hungar.}, {\bf 54}(2017), 489--508.

\bibitem{ZC} H. Zou; D. Mosi\'c and J. Chen, Generalized Drazin invertibility of the product and sum of two elements in a Banach algebra and its applications, {\it Turk. J. Math.}, {\bf 41}(2017), 548--563.

\end{thebibliography}
\end{document}